\documentclass[10pt]{amsart}
\usepackage{amsmath,amssymb,amsbsy,amsfonts,amsthm,latexsym,
                        amsopn,amstext,amsxtra,euscript,amscd,mathrsfs,color,bm,cite}
                       
\usepackage{float} 
\usepackage[english]{babel}
\usepackage{mathtools}
\usepackage{todonotes}
\usepackage{url}
\usepackage[colorlinks,linkcolor=blue,anchorcolor=blue,citecolor=blue]{hyperref}
\usepackage{enumitem}

\def\le{\leqslant}
\def\ge{\geqslant}

\newtheorem{Hypothesis}{Hypothesis}
\newtheorem{theorem}{Theorem}
\newtheorem{corollary}{Corollary}
\newtheorem{lemma}{Lemma}
\newtheorem{prop}{Proposition}

\def \vol{\mathrm {Vol}}

\newcommand\cL{\mathcal{L}}

\newcommand{\Z}{\mathbb{Z}}

\newcommand{\R}{\mathbb{R}}

\usepackage{etoolbox}

\newcommand{\QQ}{\mathbb{Q}}

\def\cL{{\mathcal L}}

\def\ov\QQ{\overline{\QQ}}

\usepackage{mathtools}
\usepackage{todonotes}
\usepackage[norefs,nocites]{refcheck}

\def\({\left(}
\def\){\right)}

\title{On certain bilinear sums with modular square roots and applications}



\author{Stephan Baier}
\address{Stephan Baier,
Ramakrishna Mission Vivekananda Educational and Research Institute, Department of Mathematics, G. T. Road, PO Belur Math, Howrah, West Bengal 711202, India}
\email{stephanbaier2017@gmail.com}

\begin{document}
\maketitle
\begin{abstract} We extend bounds on additive energies of modular square roots by Dunn, Kerr, Shparlinski, Shkredov and Zaharescu and apply these results to obtain bounds on certain bilinear exponential sums with modular square roots. From here, we make partial progress on the large sieve for square moduli.  
\end{abstract}

\tableofcontents

\section{Introduction and main results}
{\bf Notations.} 
\begin{itemize}
\item Throughout this article, following usual custom, we assume that $\varepsilon$ is an arbitrarily small positive number. All implied constants are allowed to depend on $\varepsilon$. 
\item For a real number $x$ and natural number $r$, we set 
$$
e(x):=e^{2\pi i x} \quad \mbox{and} \quad e_r(x):=e\left(\frac{x}{r}\right). 
$$
\item Finite sequences $(\alpha_l)$,$(\beta_m)$,... are abbreviated by bold letters $\boldsymbol{\alpha}$,$\boldsymbol{\beta}$,.... 
\item The functions $\tau(n)$ and $\omega(n)$ denote the numbers of divisors and prime divisors of $n\in \mathbb{N}$, respectively. The function $\mu(n)$ is the M\"obius function.  
\item We denote the greatest common divisor and the least common multiple of two integers $a,b$, not both equal to zero, by $(a,b)$ and $[a,b]$, respectively. 
\item The notation $p^k||r$ means that $p^k$ is the largest power dividing $r$ of a prime $p$. 
\item The symbols $\mathbb{P}$, $\mathbb{N}$, $\mathbb{Z}$, $\mathbb{R}$ and $\mathbb{C}$ stand for the sets of prime numbers, natural numbers, integers, real numbers and complex numbers, respectively. We denote by $\mathbb{R}_{\ge 0}$ the set of non-negative real numbers and by $\mathbb{R}_{>0}$ the set of positive real numbers.   
\item The distance of $x\in \mathbb{R}$ to the nearest integer is denoted as $||x||$.
\item The functions $\lfloor x \rfloor$ and $\lceil x \rceil$ are the floor and ceiling function, respectively.
\item For functions $f:\mathcal{D}\rightarrow \mathbb{C}$ and $g:\mathcal{D}\rightarrow \mathbb{R}_{>0}$, the notations $f(x)=O(g(x))$ and $f(x)\ll g(x)$ indicate that there is a constant $C>0$ such that $|f(x)|\le Cg(x)$ for all $x\in \mathcal{D}$, and the notation $f(x)\asymp g(x)$ indicates that there are constants $C_2>C_1>0$ such that  
$C_1g(x)\le f(x)\le C_2g(x)$ for all $x\in \mathcal{D}$. 
\end{itemize}

The starting point of this article is the following best known result on the large sieve for square moduli, proved in \cite{BaiZhao1} and recently revisited in \cite{BaiLS}.

\begin{theorem} \label{bestknownls} Let $Q,N\ge 1$, $M\in \mathbb{R}$ and $(a_n)_{M<n\le M+N}$ be any sequence of complex numbers.  Then
\begin{equation} \label{thebound}
\begin{split}
& 
\sum\limits_{q\le Q}\sum\limits_{\substack{a=1\\ (q,a)=1}}^{q^2} \left|\sum\limits_{M<n\le M+N} a_ne\left(\frac{na}{q^2}\right)\right|^2\\
\ll & (QN)^{\varepsilon}\left(Q^3+N+\min\left\{Q^2N^{1/2},Q^{1/2}N\right\}\right)Z,
\end{split}
\end{equation}
where 
$$
Z:=\sum\limits_{M<n\le M+N} |a_n|^2.
$$
\end{theorem} 

We keep the above definition of $Z$ throughout the sequel.

The above result has found a lot of arithmetic applications (see \cite{BaiZhao},  \cite{BPS}, \cite{BFKS},  \cite{Mat}, \cite{Mer}, \cite{SZ},  \cite{TuPa}). For a more detailed history of results related to the large sieve for square moduli, see \cite{BaiLS}.  A conjecture of Zhao \cite{Zhao1} predicts that the left-hand side of \eqref{thebound} should be bounded by $\ll (QN)^{\varepsilon}(Q^3+N)Z$.  In the said article \cite{BaiLS}, we derived a conditional improvement of \eqref{thebound}. Of particular interest was the "critical" point $N=Q^3$ for which \eqref{thebound} yields an estimate by $\ll Q^{1/2+\varepsilon}NZ$. There has been no unconditional improvement of this bound for the last 20 years. However, we proved in \cite{BaiLS} that the exponent $1/2$ therein can be replaced by $1/2-\eta$ with $\eta=1/135$ under the assumption of a set of certain reasonable hypotheses on additive energies of modular square roots. In this article, we do not make unconditional progress either but streamline the arguments and replace the said set of hypotheses by a single simpler and more tractable hypothesis. While our motivation comes from the large sieve for square moduli, our main result in this article is Theorem \ref{bilinearbound} below, which is unconditional and may have other applications. 

As described in \cite{BaiLS}, the large sieve for square moduli relates to bilinear sums of the form 
\begin{equation} \label{biligoal}
\sum\limits_{|l|\le L}  \sum\limits_{\substack{1\le m\le M\\ k^2\equiv jm\bmod{r}}} \alpha_l\beta_m e_r\left(lk\right),
\end{equation}
where $r\in \mathbb{N}$ and $(r,j)=1$.
What we achieve {\it unconditionally} in this article is to establish non-trivial bounds for the bilinear sums in 
\eqref{biligoal} in situations when both $L$ and $M$ are much smaller than the square root of the modulus $r$. We refer to a solution $k$ of the congruence $k^2\equiv s\bmod{r}$ as a modular square root of $s$ modulo $r$. By abuse of notation, we denote by $\sqrt{s}$ the {\it collection} of all modular square roots of $s$ modulo $r$, if existent. Thus, we are interested in sums of the form
$$
\Sigma(r,j,L,M,\boldsymbol{\alpha},\boldsymbol{\beta}):=\sum\limits_{|l|\le L} \sum\limits_{1\le m\le M} \alpha_l\beta_m e_r\left(l\sqrt{jm}\right),
$$  
where $\sqrt{jm}$ should not be mistaken for the ordinary square root of a positive real number. We point out that bilinear sums of the shape 
$$
\sum\limits_{l\in \mathcal{X}} \sum\limits_{m\in \mathcal{Y}} \alpha_l\beta_m e_r\left(hlm^t\right)
$$
with $r$ prime, $\mathcal{X}$, $\mathcal{Y}\in \subset \mathbb{F}_r^{\times}$ and $t\in \mathbb{Z}$ where considered in \cite{BagShp}. The case of $t=-1$ (Kloosterman fractions) was elaborated on in \cite{KSXW}. Bounds for bilinear sums of the form 
$$
\sum\limits_{1\le m\le M} \sum\limits_{1\le n\le N} \beta_m\gamma_n e_r\left(h\sqrt{jmn}\right)
$$
with $r$ prime were derived and applied in \cite{DKSZ}, \cite{KSSZ}, \cite{SSZ22} and \cite{SSZ}. 

More generally, we shall establish bounds for sums of the form
\begin{equation} \label{Sigmadeff}
\Sigma(r,j,L,M,\boldsymbol{\alpha},\boldsymbol{\beta},f):=\sum\limits_{|l|\le L}\sum\limits_{1\le m\le M} \alpha_l\beta_m e_r\left(l\sqrt{jm}\right)e(lf(m)),
\end{equation}
where $f: [1,M]\rightarrow \mathbb{R}$ is a continuously differentiable function such that $g(x)=e(lf(x))$ does not oscillate too wildly if $|l|\le L$. These type of sums come up in connection with the large sieve for square moduli in \cite{BaiLS}.  They contain two oscillating terms, an arithmetic term $e_r\left(l\sqrt{jm}\right)$ and an analytic term $e(lf(m))$, both being linear in $l$. Our main result is the following.

\begin{theorem} \label{bilinearbound} Suppose that $r,j\in \mathbb{N}$, $(r,j)=1$ and $1\le L,M\le r$. Let $f:[1,M]\rightarrow \mathbb{R}$ be a continuously differentiable function such that $|f'(x)|\le F$ on $[1,M]$, where 
$F\le L^{-1}$. Let $\boldsymbol{\alpha}=(\alpha_l)_{|l|\le L}$ and $\boldsymbol{\beta}=(\beta_m)_{1\le m\le M}$ be any finite sequences of complex numbers and suppose that $H\in \mathbb{N}$ satisfies 
\begin{equation} \label{Hrange}
1\le H\le \min\left\{\frac{1}{LF},M\right\}.
\end{equation}
Then we have   
\begin{equation} \label{generalbilinearbound}
\begin{split}
&\Sigma(r,j,L,M,\boldsymbol{\alpha},\boldsymbol{\beta},f)\\   \ll & 
\left(H^{-1/2}L^{1/2}M+ H^{1/4}L^{1/4}M+H^{-1/4}L^{1/4}M^{3/4}r^{1/4}\right)||\boldsymbol{\alpha}||_2||\boldsymbol{\beta}||_{\infty} r^{\varepsilon},
\end{split}
\end{equation}
where 
$$
||\boldsymbol{\alpha}||_2:=\left(\sum\limits_{|l|\le L} |\alpha_l|^2\right)^{1/2} \quad \mbox{and} \quad ||\boldsymbol{\beta}||_{\infty}:=\max\limits_{1\le m\le M} |\beta_m|. 
$$
\end{theorem}

Taking $f\equiv 0$ and $H=\lfloor r^{1/2}M^{-1/2} \rfloor$ if $r^{1/3}\le M\le r$ in Theorem \ref{bilinearbound} implies the following.

\begin{corollary} \label{bilinearcor} Suppose that $r,j\in \mathbb{N}$, $(r,j)=1$, $1\le L\le r$ and  $r^{1/3}\le M\le r$. Then for any finite sequences $\boldsymbol{\alpha}=(\alpha_l)_{|l|\le L}$ and $\boldsymbol{\beta}=(\beta_m)_{1\le m\le M}$ of complex numbers, we have   
\begin{equation} \label{specialbilinearbound} 
\Sigma(r,j,L,M,\boldsymbol{\alpha},\boldsymbol{\beta})
\ll \left(L^{1/2}M^{5/4}r^{-1/4}+L^{1/4}M^{7/8}r^{1/8}\right)||\boldsymbol{\alpha}||_2||\boldsymbol{\beta}||_{\infty}r^{\varepsilon}.
\end{equation}
\end{corollary}

This is stronger than the trivial bound 
$$
\Sigma(r,j,L,M,\boldsymbol{\alpha},\boldsymbol{\beta})\ll L^{1/2}M||\boldsymbol{\alpha}||_2||\boldsymbol{\beta}||_{\infty}
$$
if 
$$
r^{1/3}\le M\le r^{1-4\varepsilon} \quad \mbox{and} \quad L^2M\ge r^{1+8\varepsilon}.
$$
In particular, we obtain a nontrivial bound if $r^{1/3+3\varepsilon}\le L,M\le r^{1-\varepsilon}$. We do not get a non-trivial bound along these lines if both $L$ and $M$ are less than or equal to $r^{1/3}$.

To handle the above bilinear sums, we extend bounds on additive energies of modular square roots derived in \cite{KSSZ}. The authors of \cite{DKSZ} and  \cite{KSSZ} considered energies of the form
$$
T_{2}(r,j,M):=\sum\limits_{\substack{1\le m_1,m_2,m_3,m_4\le M\\ \sqrt{jm_2}-\sqrt{jm_3}\equiv \sqrt{jm_3}-\sqrt{jm_4}\bmod{r}}} 1
$$ 
and 
$$
T_{4}(r,j,M):=\sum\limits_{\substack{1\le m_1,...,m_8\le M\\ \sqrt{jm_1}-\sqrt{jm_2}+\sqrt{jm_3}-\sqrt{jm_4}\equiv \sqrt{jm_5}-\sqrt{jm_6}+\sqrt{jm_7}-\sqrt{jm_8}\bmod{r}}} 1
$$ 
for the situation when $r$ is a prime, $(r,j)=1$ and $M\ge 1$. They proved the following result. 

\begin{theorem} \label{ksszbounds} Suppose that $r\in \mathbb{P}$, $j\in \mathbb{N}$, $(r,j)=1$ and $1\le M\le r$. Then 
\begin{equation} \label{T2bound}
T_2(r,j,M)\ll \left(\frac{M^{7/2}}{r^{1/2}}+M^2\right)r^{\varepsilon}
\end{equation}
and 
\begin{equation} \label{T4bound}
T_4(r,j,M)\ll M^4\left(\frac{M^{7/2}}{r^{1/2}}+M^2\right)r^{\varepsilon}.
\end{equation}
\end{theorem}

\begin{proof} This can be found in \cite[Theorem 1.1 and (1.5)]{KSSZ}. \end{proof}

Here we consider restricted energies of the form 
\begin{equation} \label{E2def}
E_2(r,j,M,H):=\sum\limits_{\substack{1\le m_1,m_2,m_3,m_4\le M\\ 1\le |m_i-m_{i+1}|\le H \text{ for } i=1,3\\ \sqrt{jm_1}-\sqrt{jm_2}\equiv \sqrt{jm_3}-\sqrt{jm_4}\bmod{r}}} 1
\end{equation}
and 
\begin{equation} \label{E4def}
E_4(r,j,M,H):=\sum\limits_{\substack{1\le m_1,...,m_8\le M\\ 1\le |m_{i}-m_{i+1}|\le H \text{ for } i=1,3,5,7\\ \sqrt{jm_1}-\sqrt{jm_2}+\sqrt{jm_3}-\sqrt{jm_4}\equiv \sqrt{jm_5}-\sqrt{jm_6}+\sqrt{jm_7}-\sqrt{jm_8}\bmod{r}}} 1
\end{equation}
for the more general situation when $r$ is an arbitrary natural number, $(r,j)=1$ and $1\le H\le M\le r$. Following the treatment in \cite{KSSZ}, we prove the following.

\begin{theorem} \label{energybounds} Suppose that $r,j\in \mathbb{N}$, $(r,j)=1$ and $1\le H\le M\le r$. Then 
\begin{equation} \label{E2bound}
E_2(r,j,M,H)\ll \left(\frac{H^3M^2}{r}+HM\right)r^{\varepsilon}
\end{equation}
and 
\begin{equation} \label{E4bound}
E_4(r,j,M,H)\ll H^2M^2\left(\frac{H^3M^2}{r}+HM\right)r^{\varepsilon}.
\end{equation}
\end{theorem} 

We note that the lower bound $|m_i-m_{i+1}|\ge 1$ for $i$ odd in \eqref{E2def} and \eqref{E4def} is essential. If the case $m_i=m_{i+1}$ for $i$ odd were included, we would trivially have $E_2\gg M^2$. As a consequence, the term $HM$ in \eqref{E2bound} and \eqref{E4bound} would need to be replaced by $M^2$. (In subsection \ref{E4estimation}, we will prove that $E_4(r,j,M,H)\ll \left(r^{\varepsilon}HM\right)^2 E_2(r,j,M,H)$.)
  
In the case when $r$ is a prime and $H=M$, the above bounds \eqref{E2bound} and \eqref{E4bound} follow from Theorem \ref{ksszbounds}. 

In \cite{SSZ} it was proved that for, in a sense, almost all primes $r$, all $j$ coprime to $r$ and all $M$ in the range $1\le M\le r$, the bound
\begin{equation} \label{expectT2}
T_2(r,j,M)\ll \left(\frac{M^{4}}{r}+M^2\right)r^{\varepsilon}
\end{equation}
holds. Following heuristic arguments in \cite{BaiLS}, this bound should give the true order of magnitude of $T_2(r,j,M)$ even if $r$ is not restricted to primes. Similarly, it should be expected that 
\begin{equation} \label{expectT4}
T_4(r,j,M)\ll \left(\frac{M^{8}}{r}+M^4\right)r^{\varepsilon}
\end{equation}
holds in general. Using similar arguments, we may expect the true orders of magnitudes of the energies $E_2$ and $E_4$ to be
\begin{equation} \label{expectE2} 
E_2(r,j,M,H)\ll \left(\frac{H^2M^{2}}{r}+HM\right)r^{\varepsilon}
\end{equation}
and 
\begin{equation} \label{expectE4} 
E_4(r,j,M,H)\ll \left(\frac{H^4M^{4}}{r}+H^2M^2\right)r^{\varepsilon}.
\end{equation}
Here the terms $H^2M^2/r$ and $H^4M^4/r$ come from the probability $1/r$ of the congruences in \eqref{E2def} and \eqref{E4def} to hold for randomly picked $m_i$, and the terms $HM$ and $H^2M^2$ majorize the diagonal contributions coming from the trivial congruences 
$$
\sqrt{jm_1}-\sqrt{jm_2}\equiv \sqrt{jm_1}-\sqrt{jm_2}\bmod{r}
$$
and
$$
\sqrt{jm_1}-\sqrt{jm_2}+\sqrt{jm_3}-\sqrt{jm_4}\equiv \sqrt{jm_1}-\sqrt{jm_2}+\sqrt{jm_3}-\sqrt{jm_4}\bmod{r},
$$
respectively. We mention that \eqref{expectE2} would give a stronger bound of 
\begin{equation*} \label{gen1}
\begin{split}
&  \sum\limits_{|l|\le L} \sum\limits_{1\le m\le M} \alpha_l\beta_m e_r\left(l\sqrt{jm}\right)e(lf(m))\\ \ll & 
\left(H^{-1/2}L^{1/2}M+ L^{1/4}M+H^{-1/4}L^{1/4}M^{3/4}r^{1/4}\right)||\boldsymbol{\alpha}||_2||\boldsymbol{\beta}||_{\infty} r^{\varepsilon}
\end{split}
\end{equation*}
in place of \eqref{generalbilinearbound}. Taking $f\equiv 0$ and $H=M$, this implies a stronger bound of 
\begin{equation*} \label{spec1}
\begin{split}
& \sum\limits_{|l|\le L} \sum\limits_{1\le m\le M} \alpha_l\beta_m e_r\left(l\sqrt{jm}\right)\\
\ll & 
\left(L^{1/2}M^{1/2}+ L^{1/4}M+L^{1/4}M^{1/2}r^{1/4}\right)||\boldsymbol{\alpha}||_2||\boldsymbol{\beta}||_{\infty} r^{\varepsilon}
\end{split}
\end{equation*}
in place of \eqref{specialbilinearbound}, which is non-trivial if 
\begin{equation*} \label{nontrivcond2}
M,L\ge r^{4\varepsilon} \quad \mbox{and} \quad M^2L\ge r^{1+4\varepsilon}.
\end{equation*}
Consequently, even under the hypothetical bound \eqref{expectE2} we still do not get a non-trivial bound if both $M$ and $L$ are less than or equal to $r^{1/3}$.
 
In \cite[Theorem 1.2]{KSSZ}, the authors sharpened the above unconditional bound \eqref{T4bound} for $T_4(r,j,M)$ if $M\le r^{1/16-\varepsilon}$, saving an extra factor of $M^{\nu}$ for some $\nu>0$. Considering the expected bound \eqref{expectT4}, an improvement of \eqref{T4bound} should be possible in a much wider range of $M$. Similarly, one may conjecture that the bound \eqref{E4bound} for $E_4(r,j,M,H)$ can be improved. 
However, it seems hard to establish the expected bounds \eqref{expectE2} and \eqref{expectE4} in full generality. Here we propose the following weaker hypothesis on $E_4(r,j,M,H)$ which is consistent with \eqref{expectE4} and saves an extra factor of $H^{\nu}$ beyond the unconditional bound \eqref{E4bound}. 

\begin{Hypothesis} \label{hyp} There is $\nu>0$ such that 
\begin{equation} \label{E4conj}
E_4(r,j,M,H)\ll H^{2-\nu}M^2\left(\frac{H^3M^2}{r}+HM\right)r^{\varepsilon}
\end{equation}
whenever $r,j\in \mathbb{N}$, $(r,j)=1$ and $1\le H\le M\le r$.
\end{Hypothesis}

Under this hypothesis, we will establish the following result on our bilinear exponential sums.

\begin{theorem} \label{bilinearbound2} Suppose $\nu$ is a positive real number such that \eqref{E4conj} holds whenever $r,j\in \mathbb{N}$, $(r,j)=1$ and $1\le H\le M\le r$. Then, under the conditions in Theorem \ref{bilinearbound}, we have 
\begin{equation}
\begin{split}
\label{Hypobound}
&
\Sigma(r,j,L,M,\boldsymbol{\alpha},\boldsymbol{\beta},f)\\ 
\ll & \left(
H^{-1/2}L^{1/2}M+H^{1/8-\nu/8}L^{3/8}M+H^{-1/8-\nu/8}L^{3/8}M^{7/8}r^{1/8}\right)
|| \boldsymbol{\alpha}||_2||\boldsymbol{\beta}||_{\infty}r^{\varepsilon}.
\end{split}
\end{equation}
\end{theorem}  

Taking $f\equiv 0$ and $H=M$ in Theorem \ref{bilinearbound2} implies the following.

\begin{corollary} \label{bilinearcor2} Suppose $\nu$ is a positive real number such that \eqref{E4conj} holds whenever $r,j\in \mathbb{N}$, $(r,j)=1$ and $1\le H= M\le r$. Then, under the conditions in Corollary \ref{bilinearcor}, we have 
\begin{equation*} \label{specialbilinearbound2} 
\begin{split}
& \Sigma(r,j,L,M,\boldsymbol{\alpha},\boldsymbol{\beta})\\ 
\ll & \left(
L^{1/2}M^{1/2}+L^{3/8}M^{9/8-\nu/8}+L^{3/8}M^{3/4-\nu/8}r^{1/8}\right) ||\boldsymbol{\alpha}||_2||\boldsymbol{\beta}||_{\infty}r^{\varepsilon}.
\end{split}
\end{equation*}
\end{corollary}
This is non-trivial if 
$$
M\ge r^{2\varepsilon}, \quad L\ge M^{1-\nu}r^{8\varepsilon}\quad \mbox{and} \quad M^{2+\nu}L\ge r^{1+8\varepsilon}. 
$$
Hence, taking $\varepsilon$ small enough, we here get a non-trivial bound in the situation when $M=L=r^{1/3}$, which constitutes conditional progress.  Theorem \ref{bilinearbound2} also allows us to establish  the following conditional improvement of the large sieve inequality \eqref{thebound} for square moduli at the critical point $N=Q^3$.

\begin{theorem} \label{lsimpro} Suppose Hypothesis \ref{hyp} holds. Then there exists $\eta>0$ such that 
\begin{equation} \label{lssq} 
\sum\limits_{q\le Q}\sum\limits_{\substack{a=1\\ (q,a)=1}}^{q^2} \left|\sum\limits_{M<n\le M+N} a_ne\left(\frac{na}{q^2}\right)\right|^2\ll Q^{1/2-\eta}NZ
\end{equation}
whenever $N=Q^3\ge 1$.
\end{theorem}  

This greatly simplifies the hypotheses on additive energies used in \cite{BaiLS}, replacing them by a single hypothesis which might be in reach of existing methods. 

It would be possible to work out a precise relation between $\eta$ and $\nu$ in \eqref{lssq} and Hypothesis \ref{hyp}, but this involves lengthy calculations which we have not carried out in this article to keep things as simple as possible. \\ \\
{\bf Acknowledgements.} The author would like to thank the anonymous referees for valuable comments. He also thanks the Ramakrishna Mission Vivekananda Educational and Research Institute for excellent working conditions. 

\section{Additive energies of modular square roots}
To derive our energy bound \eqref{E2bound} for $E_2$ in Theorem \ref{energybounds}, we proceed closely along the lines of \cite[Section 3]{KSSZ} in which the bound \eqref{T2bound} for $T_2$ was established. The energy bound \eqref{E4bound} for $E_4$ will then follow by a simple argument. At some places, we will even keep the wording in \cite[Section 3]{KSSZ}. This method works via algebraic transformations and geometry of numbers. However, since in our case, the modulus $r$ is no longer restricted to primes and there are additional restrictions on differences of summation variables, some extensions of the arguments in \cite[Section 3]{KSSZ} are necessary. The following subsection provides results on lattices which we will use in the course of our proof of \eqref{E2bound}. 

\subsection{Preliminaries on lattices}
For general background on lattices, see \cite[Chapter 3]{TaoVu}.
Let a lattice $\Gamma\subseteq \R^{n}$ of rank $n$ be given. We denote the volume of the quotient $\mathbb{R}^n/\Gamma$ as $\vol(\R^n/\Gamma)$. For a convex body $\mathcal{B}$, we denote its volume by $\vol(\mathcal{B})$ and the successive minima of $\Gamma$ with respect to $\mathcal{B}$ as $\lambda_1\le \lambda_2\le ...\le \lambda_n$. Recall that
$$
\lambda_i :=\inf\{\lambda >0:~\lambda \mathcal{B} \text{ contains $i$ linearly independent elements of } \Gamma\}.
$$

We shall use Minkowski's second theorem. 
  
\begin{prop}[Minkowski]
\label{Minkowski}
Suppose $\Gamma \subseteq \R^{n}$ is a lattice of rank $n$, $\mathcal{B}\subseteq \R^{n}$ a symmetric convex body and let $\lambda_1,\ldots,\lambda_n$ denote the successive minima of $\Gamma$ with respect to $\mathcal{B}$. Then we have
$$
\frac{1}{2^n}\cdot \frac{ \vol (\mathcal{B})}{\vol(\R^n/\Gamma)}\le \frac{1}{\lambda_1\cdots\lambda_n}\le \frac{n!}{2^n}\cdot \frac{ \vol (\mathcal{B})}{\vol(\R^n/\Gamma)}.
$$
\end{prop}

\begin{proof} See \cite[Theorem~3.30]{TaoVu}. \end{proof} 

We shall also use the following bound for the number of lattice points inside a symmetric convex body. 

\begin{prop}[Betke, Hank, Wills]
\label{latticepoints}
Suppose $\Gamma \subseteq \R^{n}$ is a lattice of rank $n$, $\mathcal{B}\subseteq \R^{n}$ a symmetric convex body and let $\lambda_1,\ldots,\lambda_n$ denote the successive minima of $\Gamma$ with respect to $\mathcal{B}$. Then we have
$$\sharp\(\Gamma \cap \mathcal{B}\) \le \prod_{i=1}^{n}\left(\frac{2i}{\lambda_i}+1\right).$$
\end{prop}

\begin{proof} 
This is \cite[Proposition~2.1]{BHW}.
\end{proof}

\subsection{Bounding $E_2(r,j,M,H)$}
\subsubsection{General strategy}
Recall that $(r,j)=1$. We may write
\begin{equation} \label{startingpoint}
E_2(r,j,M,H)=\sum\limits_{d=1}^r I(d)^2, 
\end{equation}
where 
$$
I(d):=\sum\limits_{\substack{1\le m_1,m_2\le M\\ 1\le |m_1-m_2|\le H\\ \sqrt{jm_1}-\sqrt{jm_2}\equiv d\bmod{r}}} 1.
$$
Below we will bound the first moment of $I(d)$ by  
\begin{equation} \label{firstmoment}
\sum\limits_{d=1}^r I(d) \ll r^{\varepsilon}HM
\end{equation}
using rather simple arguments. 
Our general strategy for an estimation of the second moment is to split the sum over $d$ in \eqref{startingpoint} into suitable subsums 
$$
\Sigma(\mathcal{M}):=\sum\limits_{d\in \mathcal{M}} I(d)^2,
$$
where $\mathcal{M}\subseteq \{1,...,r\}$. Some of them will be estimated directly, whereas for others we derive an individual bound for $I(d)$ if $d\in \mathcal{M}$ and estimate $\Sigma(\mathcal{M})$ by
$$
\Sigma(\mathcal{M})\le \left(\sup\limits_{d\in \mathcal{M}} I(d)\right)\sum\limits_{d=1}^r I(d) \ll
 \left(\sup\limits_{d\in \mathcal{M}} I(d)\right)r^{\varepsilon}HM.
$$ 

\subsubsection{First moment of $I(d)$}
To estimate the sum on the left-hand side of \eqref{firstmoment}, we use two lemmas. 
 
\begin{lemma} \label{} 
For $r\in \mathbb{N}$ and $m\in \mathbb{Z}$, let $s(r;m)$ be the number of modular square roots of $m$ modulo $r$. Then
\begin{equation} \label{srm}
s(r;m)\ll r^{\varepsilon} \sqrt{(r,m)}.
\end{equation}
\end{lemma} 

\begin{proof} 
We first look at the case when $r=p^k$ is a prime power. If $(p^k,m)=p^l$, then the congruence 
\begin{equation} \label{quadcon}
x^2\equiv m\bmod{p^k}
\end{equation}
is equivalent to 
$$
\frac{x^2}{p^l} \equiv m_0\bmod{p^{k-l}} 
$$
with 
$$
m_0:=\frac{m}{p^l}, \quad  \left(m_0,p^{k-l}\right)=1.
$$
A necessary condition for this congruence to be solvable for $x$ is that $l$ is even, $p^{l/2}|x$ and $(x/p^{l/2},p^{k-l})=1$. In this case, the above congruence turns into  
$$
x_0^2 \equiv m_0\bmod{p^{k-l}}
$$ 
with 
$$
x_0:=\frac{x}{p^{l/2}}, \quad \left(x_0,p^{k-l}\right)=1,
$$
which, using Hensel's lemma, has  $O(1)$ solutions $x_0$ modulo $p^{k-l}$. Consequently, \eqref{quadcon} has $O(p^{l/2})$ solutions $x$ modulo $p^{k}$. It follows that
$$
s\left(p^k,m\right)\ll p^{l/2}=\sqrt{(p^k,m)}. 
$$
Using the Chinese remainder theorem, we deduce that 
$$
s(r;m)=\prod\limits_{p^k||r} s\left(p^k,m\right)\ll r^{\varepsilon} \sqrt{(r,m)} 
$$  
for any general modulus $r\in \mathbb{N}$ and $m\in \mathbb{Z}$. 
\end{proof}

\begin{lemma} Let $r,j\in \mathbb{N}$ with $(r,j)=1$ and $M\ge 1$. Then
\begin{equation}\label{gcdsum}
\sum\limits_{1\le m\le M} (r,jm)\ll r^{\varepsilon}M.
\end{equation}
\end{lemma}

\begin{proof}
We have 
$$
\sum\limits_{1\le m\le M} (r,jm)=\sum\limits_{1\le m\le M} (r,m)\le \sum\limits_{d|r} d\sum\limits_{\substack{1\le m\le M\\ d|m}} 1 \le \sum\limits_{d|r} d\cdot \frac{M}{d}= \tau(r)M \ll r^{\varepsilon}M.
$$
\end{proof}

Write $s(m)=s(r,jm)$. Then trivially, we have
\begin{equation} \label{trivialbound}
\sum\limits_{d=1}^r I(d) = \sum\limits_{\substack{1\le m_1,m_2\le M\\ 1\le |m_1-m_2|\le H}} s(m_1)s(m_2).
\end{equation}
Applying the arithmetic geometric mean inequality, we deduce that
\begin{equation} \label{CSapp}
\begin{split}
\sum\limits_{\substack{1\le m_1,m_2\le M\\ 1\le |m_1-m_2|\le H}} s(m_1)s(m_2) \le &\sum\limits_{\substack{1\le m_1,m_2\le M\\ 1\le |m_1-m_2|\le H}} \frac{s(m_1)^2+s(m_2)^2}{2}\\
\le & H\sum\limits_{1\le m\le M} s(m)^2.
\end{split}
\end{equation}
Plugging in \eqref{srm} and using \eqref{gcdsum}, we obtain
\begin{equation} \label{sm2}
\sum\limits_{1\le m\le M} s(m)^2\ll r^{3\varepsilon}M.
\end{equation}
Combining \eqref{trivialbound}, \eqref{CSapp} and  \eqref{sm2}, we get the claimed bound \eqref{firstmoment} upon redefining $\varepsilon$. 

\subsubsection{Algebraic transformations}
For $i=1,2$, suppose that $\sqrt{jm_i}$ is a modular square root of $jm_i$ modulo $r$. We begin by removing the square roots from the congruence 
$$
 \sqrt{jm_1}-\sqrt{jm_2}\equiv d\bmod{r}.
$$ 
Squaring this congruence 
and rearranging terms, we get
$$
j\left(m_1+m_2\right)-d^2\equiv 2\sqrt{jm_1}\sqrt{jm_2} \bmod{r}.
$$
Again, squaring this and rearranging terms gives
$$
j^2\left(m_1+m_2\right)^2-4j^2m_1m_2+d^4\equiv 2j\left(m_1+m_2\right)d^2\bmod{r}. 
$$
Making changes of variables
$$
m_1-m_2=h \quad \mbox{and} \quad m_1+m_2=m,
$$
the above simplifies into
$$
j^2h^2+d^4\equiv 2jd^2m \bmod{r}.
$$
Multiplying this by $k^2$, where $k$ is a multiplicative inverse of $j$ modulo $r$, gives 
$$
h^2+k^2d^4\equiv 2kd^2m \bmod{r}.
$$   
Note that $m$ and $h$ determine $m_1$ and $m_2$.
It follows that  
\begin{equation*} \label{IJ}
I(d)\le J(d):= \sharp\mathcal{M}(d),
\end{equation*}
where
\begin{equation} \label{Mddef}
\mathcal{M}(d):= \{(h,m)\in \mathbb{Z}^2:  1\le |h|\le H, \ 1\le m\le 4M,\ h^2+k^2d^4\equiv kd^2m \bmod{r} \}.
\end{equation}
Hence, using \eqref{startingpoint}, we have the bound 
\begin{equation} \label{E2initial}
E_{2}(r,j,M,H)\le \sum_{d=1}^r I(d)J(d).
\end{equation}

It will be beneficial to reduce the congruence 
\begin{equation} \label{congruence}
h^2+k^2d^4\equiv kd^2m \bmod{r}
\end{equation}
in \eqref{Mddef} by removing common factors of $d^2$ and $r$. Suppose that $s^2$ is the largest square dividing $r$ and $(s,d)=t$. Then \eqref{congruence} implies $t|h$ and 
\begin{equation*} \label{congruence2}
h_0^2+k^2t^2d_0^4 \equiv kd_0^2m\bmod{r_0},
\end{equation*}
where 
$$
h_0:=\frac{h}{t}, \quad d_0:=\frac{d}{t}, \quad r_0:=\frac{r}{t^2}, \quad d_0\le tr_0.
$$
Moreover, $u=(d_0^2,r_0)$ is necessarily square-free and $(u,r_0/u)=1$. It follows that $u|d_0$, $u|h_0$ and  
$$
h_1^2+k^2t^2u^2d_1^4\equiv kd_1^2m\bmod{\tilde{r}},
$$ 
where 
$$
h_1:=\frac{h_0}{u}, \quad d_1:=\frac{d_0}{u}, \quad \tilde{r}:=\frac{r_0}{u}, \quad (\tilde{r},d_1)=1, \quad d_1\le t\tilde{r}.
$$
Hence,
\begin{equation*}
I(d)\le J(d)\le J(\tilde{r},t,u;d_1):=\sharp\mathcal{M}(\tilde{r},t,u;d_1), \quad \mbox{where } d=tud_1,\ (\tilde{r},d_1)=1
\end{equation*}
and 
\begin{equation} \label{Mrtud}
\begin{split}
\mathcal{M}(\tilde{r},t,u;d_1):=\{& (h,m)\in \mathbb{Z}^2:  1\le |h|\le H/(tu), \ 1\le m\le 4M,\\ & h^2+k^2t^2u^2d_1^4\equiv kd_1^2m \bmod{\tilde{r}} \}.
\end{split}
\end{equation}
On relabeling $d_1$ as $d$, we deduce that 
\begin{equation} \label{E2initial2}
\sum\limits_{d=1}^r I(d)J(d)\le  \sideset{}{^{\prime}}\sum\limits_{\substack{\tilde{r},t,u\\ \tilde{r}t^2u=r\\ tu\le H}} \mu^2(u)\sideset{}{^{\ast}}\sum_{d=1}^{t\tilde{r}} I(tud)J(\tilde{r},t,u;d)
\end{equation}
and 
\begin{equation} \label{Idsum}
\sum\limits_{d=1}^r I(d)= \sideset{}{^{\prime}}\sum\limits_{\substack{\tilde{r},t,u\\ \tilde{r}t^2u=r}} \mu^2(u)\sideset{}{^{\ast}}\sum_{d=1}^{t\tilde{r}} I(tud),
\end{equation}
where the dash and asterisk superscripts indicate that $(\tilde{r},u)=1$ and $(\tilde{r},d)=1$, respectively, and
the restriction $tu\le H$ in the summation on the right-hand side of \eqref{E2initial2} comes from the condition $1\le |h|\le H/(tu)$ in \eqref{Mrtud}. 
This restriction will be essential in what follows. To see that equality holds in \eqref{Idsum}, note that 
$$
\{1,2,...,r\}=\bigcup\limits_{\substack{\tilde{r},t,u\\ \tilde{r}t^2u=r\\ (\tilde{r},u)=1\\ \mu^2(u)=1}} \left\{tud: 1\le d\le t\tilde{r},\ (\tilde{r},d)=1\right\}.
$$

\subsubsection{Contribution of $\tilde{r}=1$}
We first bound the contribution of $\tilde{r}=1$ to the right-hand side of \eqref{E2initial2}. If $\tilde{r}=1$, then $t^2u=r$ and $\mu^2(u)=1$, which determines $t$ and $u$ uniquely. The said contribution is vacuous if  $tu> H$. If $tu\le H$, then it is bounded by 
\begin{equation} \label{1cont}
\begin{split}
\ll &  \sum\limits_{d=1}^t I(tud)J(1,t,u;d)\\
\le& \sum\limits_{d=1}^t J(1,t,u;d)^2\\
\ll & t\left(\frac{H}{tu}\right)^2M^2\\ =& \frac{tH^2M^2}{ru}\\
\le & \frac{H^3M^2}{ru^2}\\ \le & \frac{H^3M^2}{r}.
\end{split}
\end{equation}
Hence, in the following, it suffices to consider the case when $\tilde{r}>1$, which we want to assume throughout the following.  

\subsubsection{Partitioning of the $d$-sum}
Let $\cL(\tilde{r};d)$ denote the lattice 
$$
\cL(\tilde{r};d):=\{ (x,y) \in \Z^2 : x\equiv kd^2 y \bmod{\tilde{r}} \},
$$
$\mathcal{B}(\tilde{H})$ the convex body 
\begin{equation} \label{convex}
\mathcal{B}(\tilde{H}):=\{(x,y)\in \R^2 : |x|\le \tilde{H}^2,\  |y|\le 4M \} \quad \mbox{with } \tilde{H}:=\frac{H}{tu},
\end{equation}
and let $\lambda_1(\tilde{H},\tilde{r};d),\lambda_2(\tilde{H},\tilde{r};d)$ denote the first and second successive minima of $\cL(\tilde{r};d)$ with respect to $\mathcal{B}(\tilde{H})$. We will use the fact that  
\begin{equation*} \label{imp}
(h,m),(\tilde{h},\tilde{m})\in \mathcal{M}(\tilde{r},t,u;d) \Longrightarrow (h^2-\tilde{h}^2,m-\tilde{m})\in \cL(\tilde{r};d)\cap \mathcal{B}(\tilde{H}). 
\end{equation*}
In particular, if $\mathcal{M}(\tilde{r},t,u;d)$ is non-empty and $(h_d,m_d)\in \mathcal{M}(\tilde{r},t,u;d)$, then 
\begin{equation} \label{Jdrew}
\begin{split}
J(\tilde{r},t,u;d)=\sharp\{ (h,m)\in \mathbb{Z}  : \ & 1\le |h|\le \tilde{H},\ 1\le m\le 4M,\\ 
&  (h^2-h_d^2,m-m_d)\in \cL(\tilde{r};d)\cap \mathcal{B}(\tilde{H})\}.
\end{split}
\end{equation}

Similarly as in \cite[section 3]{KSSZ}, we partition the $d$-summation in \eqref{E2initial2} according to the sizes of $\lambda_1(\tilde{H},\tilde{r};d)$ and  $\lambda_2(\tilde{H},\tilde{r};d)$ to get 
\begin{equation}\label{dsumsplit}
\sideset{}{^{\ast}}\sum_{d=1}^{t\tilde{r}} I(tud)J(\tilde{r},t,u;d)= S_0(\tilde{r},t,u)+S_1(\tilde{r},t,u)+S_2(\tilde{r},t,u),
\end{equation}
where 
$$
S_0(\tilde{r},t,u) :=\sideset{}{^\ast}\sum_{\substack{d=1\\ \lambda_1(\tilde{H},\tilde{r};d)>1}}^{t\tilde{r}} I(tud)J(\tilde{r},t,u;d),\qquad
S_1(\tilde{r},t,u) :=\sideset{}{^\ast}\sum_{\substack{d=1\\ \lambda_1(\tilde{H},\tilde{r};d)\le 1 \\ \lambda_2(\tilde{H},\tilde{r};d)>1}}^{t\tilde{r}} I(tud)J(\tilde{r},t,u;d),
$$
$$
S_2(\tilde{r},t,u) :=\sideset{}{^\ast}\sum_{\substack{d=1\\ \lambda_2(\tilde{H},\tilde{r};d)\le 1}}^{t\tilde{r}}I(tud)J(\tilde{r},t,u;d).
$$

\subsubsection{The case $\lambda_1>1$}
If $\lambda_1(\tilde{H},\tilde{r};d)>1$ then $I(tud)\le J(\tilde{r},t,u;d)\le 1$ using \eqref{Jdrew}, and so 
$$
I(tud)J(\tilde{r},t,u;d)=I(tud),
$$ 
implying
\begin{equation} \label{S0bound}
\sideset{}{^{\prime}}\sum\limits_{\substack{\tilde{r},t,u\\ \tilde{r}t^2u=r\\ tu\le H}} \mu^2(u)S_0(\tilde{r},t,u)\le\sideset{}{^{\prime}} \sum\limits_{\substack{\tilde{r},t,u\\ \tilde{r}t^2u=r}} \mu^2(u)\sideset{}{^\ast}\sum_{d=1}^{t\tilde{r}} I(tud)
\ll r^{\varepsilon}HM,
\end{equation} 
where we have used \eqref{firstmoment} and \eqref{Idsum}. 

\subsubsection{The case $\lambda_1\le 1$, $\lambda_2>1$}
Next we consider the contribution of the cases when $\lambda_1(\tilde{H},\tilde{r};d)\le 1$ and $\lambda_2(\tilde{H},\tilde{r};d)>1$. We first observe that
\begin{equation} \label{S1ob}
S_1(\tilde{r},t,u) \le t\sideset{}{^\ast}\sum_{\substack{d=1\\ \lambda_1(\tilde{H},\tilde{r};d)\le 1 \\ \lambda_2(\tilde{H},\tilde{r};d)>1}}^{\tilde{r}} J(\tilde{r},t,u;d)^2
\end{equation}
using $I(tud)\le  J(\tilde{r},t,u;d)$ and the $\tilde{r}$-periodicity of $ J(\tilde{r},t,u;d)$ in $d$. Suppose further that
 $\mathcal{M}(\tilde{r},t,u;d)$ is non-empty and $(h_d,m_d)\in \mathcal{M}(\tilde{r},t,u;d)$. The conditions $\lambda_1(\tilde{H},\tilde{r};d)\le 1$ and $\lambda_2(\tilde{H},\tilde{r};d)>1$ imply that there is a shortest non-zero vector $\boldsymbol{v}$  in $\mathcal{L}(\tilde{r};d)\cap \mathcal{B}(\tilde{H})$ such that all other elements of this set are multiples of $\boldsymbol{v}$. Suppose that $\boldsymbol{v}=(a,b)$ (meaning a pair - not to be confused with the gcd of $a$ and $b$). Since $\tilde{r}>1$, $(\tilde{r},kd^2)=1$ and $\tilde{H},M\ge 1$, we necessarily have $ab\not=0$. (Otherwise, since $(\tilde{r},kd^2)=1$, $\mathcal{L}(\tilde{r};d)\cap \mathcal{B}(\tilde{H})$ would contain the vector $(0,\tilde{r})$ or the vector $(\tilde{r},0)$, but then it must contain a vector $(1,m)$ with $0<|m|<\tilde{r}$ or a vector $(h,1)$ with $0<h<\tilde{r}$ as well.) Again using \eqref{Jdrew}, it follows that 
\begin{equation*}
\begin{split}
J(\tilde{r},t,u;d)= \sharp\Bigg\{ & (h,m)\in \mathbb{Z}^2: 1\le |h|\le \tilde{H},\ 1\le m\le 4M,\\ &
\frac{h^2-h_d^2}{a}=\frac{m-m_d}{b}\in \mathbb{Z}\Bigg\}.
\end{split}
\end{equation*}

For any given $c$ coprime to $\tilde{r}$, using the Chinese remainder theorem and Hensel's lemma, we see that the number of solutions to the congruence $kd^2\equiv c\bmod{\tilde{r}}$ is bounded by 
$$
\ll 2^{\omega(\tilde{r})}\ll r^{\varepsilon}. 
$$
Therefore, as $d$ runs over the reduced residue classes modulo $\tilde{r}$, each lattice $\mathcal{L}(\tilde{r};d)$ repeats at most $O(r^{\varepsilon})$ times. Moreover, given $\boldsymbol{v}=(a,b)$ as above, there exist at most gcd$(a,b)$ reduced residue classes $c\bmod {\tilde{r}}$ such that $a\equiv cb\bmod{\tilde{r}}$. Hence, for any given vector $(a,b)\in \mathcal{B}(\tilde{H})\cap \mathbb{Z}^2$ with $ab\not=0$, the number of integers $d\in \{1,...,\tilde{r}\}$ such that $(\tilde{r},d)=1$ and $\mathcal{L}(\tilde{r},d)$ contains $(a,b)$ is $O\left( r^{\varepsilon}\mbox{gcd}(a,b)\right)$. In view of this, we deduce from \eqref{S1ob} that 
\begin{equation}
\begin{split} 
\label{S1bound}
S_1(\tilde{r},t,u)\ll r^{\varepsilon} t\sum_{\substack{(a,b)\in \mathcal{B}(\tilde{H})\cap \mathbb{Z}^2\\ ab\not=0}} \mbox{gcd}(a,b)\cdot K(a,b)^2,
\end{split} 
\end{equation}  
where 
\begin{equation} \label{KAB}
K(a,b):=\sharp\left\{(h,m)\in \mathbb{Z}^2 :  |h|\le \tilde{H}, \ |m|\le 4M,\  \frac{h^2-h_{a,b}^2}{a}=\frac{m-m_{a,b}}{b}\in \mathbb{Z}\right\},
\end{equation}
for some choice of integers $m_{a,b},h_{a,b}$ satisfying $|h_{a,b}|\le \tilde{H}$ and $|m_{a,b}|\le 4M$. In the following, we will estimate the right-hand side of \eqref{S1bound} in essentially the same way as \cite[last line of (3.6)]{KSSZ}. 

If $(h,m)$ is contained in the set on the right-hand side of \eqref{KAB}, then 
\begin{equation}
\label{acond}
h^2- h_{a,b}^2 \equiv 0 \bmod{|a|}
\end{equation}
and 
\begin{equation}
\label{bcond}
m-m_{a,b}\equiv 0 \bmod{|b|}.
\end{equation}
Furthermore, if one out of $h$ or $m$ is fixed then the the other 
number is defined in no more than two ways.  Write~\eqref{acond} as 
$$
(h-h_{a,b})(h+h_{a,b})\equiv 0 \bmod{|a|}.
$$
Then, recalling the definition of $\mathcal{B}(\tilde{H})$ in \eqref{convex}, we see that there are two integers $a_1,a_2$ satisfying 
$$
a_1a_2=a, \qquad |a_1|,|a_2|\le \tilde{H}^2
$$  
such that 
$$
h\equiv h_{a,b} \bmod{|a_1|}, \quad h\equiv -h_{a,b} \bmod{|a_2|}.
$$
Hence, for each fixed pair $(a_1, a_2)$, there are at most 
$$
\frac{\tilde{H}}{[a_1,a_2]}+1 \ll \frac{\tilde{H}}{|a|}\cdot (a_1,a_2)
$$
possibilities for $h$.  We deduce that
$$
K(a,b) 
 \ll  \frac{\tilde{H}}{|a|}\sum_{a_1a_2=a}(a_1,a_2).
$$
Using the Cauchy-Schwarz inequality and the well-known bound $\tau(n)\ll n^{\varepsilon}$ for the divisor function, 
it follows that
\begin{equation}
\label{firstKbound}
K(a,b)^2\ll r^{\varepsilon}\tilde{H}^{2}\sum_{a_1a_2=a}\frac{(a_1,a_2)^2}{|a|^2}.
\end{equation}
Similarly, using~\eqref{bcond}, we obtain 
\begin{equation}
\label{secondKbound}
K(a,b)\ll \frac{M}{|b|}.
\end{equation}

Combining \eqref{S1bound}, \eqref{firstKbound} and \eqref{secondKbound}, we deduce that
\begin{align*}
S_1(\tilde{r},t,u)&\ll r^{2\varepsilon}t\sum_{1\le |a|\le \tilde{H}^2}\ \sum_{1\le |b|\le 4M}(a,b) \sum_{\substack{a_1a_2=a \\ |a_1|,|a_2|\le \tilde{H}^2}}\min\left\{\frac{(a_1,a_2)^2\tilde{H}^2}{a^2},\frac{M^2}{b^2} \right\} \\
& \le r^{2\varepsilon} t\sum\limits_{1\le f\le 4M}f  \sum_{1\le a\le \tilde{H}^2/f}\ \sum\limits_{1\le b\le 4M/f}
\sum_{\substack{a_1a_2=fa \\ 1\le a_1,a_2\le \tilde{H}^2}}\min\left\{\frac{(a_1,a_2)^2\tilde{H}^2}{f^2a^2},\frac{M^2}{f^2b^2} \right\} \\
& \le r^{2\varepsilon}t\sum\limits_{1\le f\le 4M}\frac{1}{f} \sum_{ 1\le a_1,a_2\le \tilde{H}^2}\ \sum_{1\le b\le 4M}\min\left\{\frac{(a_1,a_2)^2\tilde{H}^2}{a^2_1a^2_2},\frac{M^2}{b^2} \right\} \\ 
&\ll  r^{3\varepsilon}t\sum_{1\le e \le \tilde{H}^2}\ \sum_{\substack{1\le a_1,a_2\le \tilde{H}^2 \\ (a_1,a_2)=e}}\
\sum_{ 1\le b\le 4M}
\min\left\{\frac{e^2\tilde{H}^2}{a^2_1a^2_2},\frac{M^2}{b^2} \right\} \\
& \le r^{3\varepsilon}t\sum_{1\le e \le \tilde{H}^2}\ \sum_{1\le a_1,a_2\le \tilde{H}^2/e} \ \sum_{1\le b\le 4M}\min\left\{\frac{\tilde{H}^2}{a^2_1a^2_2e^2},\frac{M^2}{b^2}\right\}, 
\end{align*} 
where $(a,b)$ here means the gcd of $a$ and $b$. 
Using our above bound for the divisor function, it follows that 
\begin{equation*}
\begin{split} 
S_1(\tilde{r},t,u)& \ll r^{4\varepsilon} t\sum_{1\le a\le \tilde{H}^4} \sum_{ 1\le b\le 4M} \min\left\{\frac{\tilde{H}^2}{a^2},\frac{M^2}{b^2} \right\}\\
& \le r^{4\varepsilon} t\(
\sum_{1\le a\le \tilde{H}^4}\sum_{\substack{1\le b\le aM/\tilde{H} }}\frac{\tilde{H}^2}{a^2}+\sum_{1\le b\le 4M} \sum_{\substack{1\le a\le b\tilde{H}/M }}\frac{M^2}{b^2}\)\\ 
& \ll r^{5\varepsilon}t\tilde{H}M=\frac{r^{5\varepsilon}HM}{u}. 
\end{split}
\end{equation*}
Hence,
\begin{equation}
\label{S1boundfinal}
\sum\limits_{\substack{\tilde{r},t,u\\ \tilde{r}t^2u=r\\ tu\le H}} \mu^2(u)S_1(\tilde{r},t,u)\ll r^{6\varepsilon}HM,
\end{equation}
using our bound for the divisor function again. 

\subsubsection{The case $\lambda_2\le 1$}
It remains to consider $S_2(\tilde{r},t,u)$. Suppose that $d$ is an integer satisfying $1\le d\le \tilde{r}$, $(\tilde{r},d)=1$ and $\lambda_2(\tilde{H},\tilde{r};d)\le 1$. 
Then for each $|h|\le \tilde{H}$, there exist at most 
$$
1+\frac{4M}{\tilde{r}}
$$ 
solutions $m$ of the congruence 
\begin{equation} \label{congruence3}
h^2+k^2t^2u^2d^4\equiv kd^2m \bmod{\tilde{r}}
\end{equation}
in \eqref{Mrtud} such that $1\le m\le 4M$. Moreover, for any two pairs $(h_1,m_1), (h_2,m_2)$ satisfying \eqref{congruence3}, we have 
$$
h_1^2-h_2^2\equiv kd^2(m_1-m_2) \bmod \tilde{r}.
$$
This implies  
\begin{equation*}
\begin{split}
J(\tilde{r},t,u;d)^2\ll & 1+ \left(1+\frac{M}{\tilde{r}}\right)\times\\ & \sharp\{(h_1,h_2,m)\in \mathbb{Z}^3:  1\le |h_1|,|h_2|\le \tilde{H}, \  |m|\le 4M,\ h_1^2\neq h_2^2, \\ & \qquad\qquad\qquad\qquad\qquad h_1^2-h_2^2\equiv kd^2m \bmod{r}\}
\end{split}
\end{equation*}
upon noting that $m_1-m_2$ and $m_2$ fix $m_1$. Since 
\begin{equation*}
\sharp\{ (h_1,h_2)\in \mathbb{Z}^2:  h_1^2-h_2^2=l\}
= \sharp\{ (h_1,h_2)\in \mathbb{Z}^2 :  (h_1-h_2)(h_1+h_2)=l\} \ll |l|^{\varepsilon}
\end{equation*}
for any $l\in \mathbb{Z}\setminus\{0\}$ by our bound for the divisor function, it follows that
\begin{equation}\label{Jbound}
J(\tilde{r},t,u;d)^2\ll 1+r^{\varepsilon}\left(1+\frac{M}{\tilde{r}}\right)\cdot \sharp\(\cL(\tilde{r};d)\cap \mathcal{B}(\tilde{H})\).
\end{equation}
Furthermore, by Propositions~\ref{Minkowski} and \ref{latticepoints}, we have 
\begin{equation}
\label{latticecount}
\sharp\(\cL(\tilde{r};d)\cap \mathcal{B}(\tilde{H})\)\ll \frac{ \vol (\mathcal{B}(\tilde{H}))}{\vol(\R^2/\cL(\tilde{r};d))}\ll \frac{\tilde{H}^2M}{\tilde{r}}.
\end{equation}
Combining \eqref{Jbound} with \eqref{latticecount}, and taking the square root, we obtain
$$
J(\tilde{r},t,u;d)\ll 1+r^{\varepsilon}\left(1+\frac{M^{1/2}}{\tilde{r}^{1/2}}\right)\cdot  \frac{\tilde{H}M^{1/2}}{\tilde{r}^{1/2}}= 1+r^{\varepsilon}\left(\frac{HM^{1/2}}{r^{1/2}u^{1/2}}+\frac{tHM}{r}\right),
$$
which implies 
\begin{equation}
\begin{split} 
\label{S2bound}
&\sideset{}{^{\prime}}\sum\limits_{\substack{\tilde{r},t,u\\ \tilde{r}t^2u=r\\ tu\le H}} \mu^2(u)S_2(\tilde{r},t,u)\\ \ll & \sideset{}{^{\prime}}\sum\limits_{\substack{\tilde{r},t,u\\ \tilde{r}t^2u=r\\ tu\le H}} \mu^2(u) r^{\varepsilon}\left(1+\frac{HM^{1/2}}{r^{1/2}u^{1/2}}+\frac{tHM}{r}\right)\sideset{}{^{\ast}} \sum_{d=1}^{t\tilde{r}} I(tud)\\
\le & r^{\varepsilon}\left(1+\frac{HM^{1/2}}{r^{1/2}}+\frac{H^2M}{r}\right)\sideset{}{^{\prime}}\sum\limits_{\substack{\tilde{r},t,u\\ \tilde{r}t^2u=r\\ tu\le H}} \mu^2(u)
\sideset{}{^{\ast}}\sum_{d=1}^{t\tilde{r}} I(tud)\\
= & r^{\varepsilon}\left(1+\frac{HM^{1/2}}{r^{1/2}}+\frac{H^2M}{r}\right)\sum_{d=1}^{r} I(d)\\
\ll  & r^{2\varepsilon}\left(HM+\frac{H^2M^{3/2}}{r^{1/2}}+\frac{H^3M^2}{r}\right),
\end{split}
\end{equation}
where we have used \eqref{firstmoment} and \eqref{Idsum} again. 

\subsubsection{Concluding the proof}
Combining \eqref{E2initial}, \eqref{E2initial2}, \eqref{1cont}, \eqref{dsumsplit}, \eqref{S0bound}, \eqref{S1boundfinal} and \eqref{S2bound}, upon redefining $\varepsilon$, we arrive at the bound 
\begin{equation}\label{finalE2bound}
E_2(r,j,M,H)\ll r^{\varepsilon}\left(\frac{H^3M^2}{r}+\frac{H^2M^{3/2}}{r^{1/2}}+HM\right).
\end{equation}
This implies the desired bound \eqref{E2bound} on $E_2(r,j,M,H)$ since the middle term is the geometric mean of the first and third terms on the right-hand side of \eqref{finalE2bound}.

\subsection{Bounding $E_4(r,j,M,H)$} \label{E4estimation}
Now we proceed to our proof of \eqref{E4bound}. We use the bound
\begin{equation*} 
\Bigg| \sum\limits_{\substack{1\le m_1,m_2\le M\\ 1\le |m_1-m_2|\le H}} e_r\left(a\left(\sqrt{jm_1}-\sqrt{jm_2}\right)\right)\Bigg|\le \sum\limits_{d=1}^r I(d) \ll r^{\varepsilon}HM,
\end{equation*}
where the first inequality holds trivially and the second inequality is \eqref{firstmoment}.
This implies 
\begin{equation*}
\begin{split}
E_4(r,j,M,H) = & \frac{1}{r} \cdot \sum\limits_{a=1}^{r} \Bigg| \sum\limits_{\substack{1\le m_1,m_2\le M\\ 1\le |m_1-m_2|\le H}} e_r\left(a\left(\sqrt{jm_1}-\sqrt{jm_2}\right)\right)\Bigg|^4\\
\ll &  \left(r^{\varepsilon}HM\right)^2\cdot \frac{1}{r} \cdot\sum\limits_{a=1}^{r} \Bigg| \sum\limits_{\substack{1\le m_1,m_2\in M\\ 1\le |m_1-m_2|\le H}} e_r\left(a\left(\sqrt{jm_1}-\sqrt{jm_2}\right)\right)\Bigg|^2\\
= & r^{2\varepsilon}H^2M^2E_2(r,j,M,H).
\end{split}
\end{equation*}
Combining the above with \eqref{E2bound}, the desired bound \eqref{E4bound} on $E_4(r,j,M,H)$ follows upon redefining $\varepsilon$ . This completes the proof of Theorem \ref{energybounds}. 

\section{Bounds on bilinear sums with modular square roots}
\subsection{Preliminaries from the theory of exponential sums}
To prove Theorems \ref{bilinearbound} and \ref{bilinearbound2}, we need two standard tools used in the theory of exponential sums: Weyl differencing and the Poisson summation formula. 

\begin{prop}[Weyl differencing] \label{Weyldif}
Let $I=(a,b]$ be an interval of length $|I|=b-a\ge 1$. Let $g: I \rightarrow \mathbb{R}$ be a function and $(\beta_m)_{m\in I}$ be a sequence of complex numbers. Then for every $H\in \mathbb{N}$ with $1\le H\le |I|$, we have the bound
$$
\left|\sum\limits_{m\in I}  \beta_me(g(m))\right|^2\ll \frac{|I|}{H}\cdot \sum\limits_{m\in I} |\beta_m|^2+\frac{|I|}{H}\cdot \sum\limits_{\substack{m_1,m_2\in I\\ 1\le |m_1-m_2|\le H}} \gamma(m_1,m_2) e(g(m_1)-g(m_2)),
$$ 
where 
\begin{equation} \label{cdef}
\gamma(m_1,m_2):= \left(1-\frac{|m_1-m_2|}{H}\right)\beta_{m_1} \beta_{m_2}.
\end{equation}
\end{prop}

\begin{proof}
This is a consequence of \cite[Lemma 2.5.]{GrKo}.
\end{proof} 

For a Schwartz class function  $\Phi:\mathbb{R} \rightarrow \mathbb{C}$, we define its Fourier transform $\hat{\Phi}:\mathbb{R} \rightarrow \mathbb{C}$ as 
$$
\hat{\Phi}(y):=\int\limits_{\mathbb{R}} \Phi(x)e(-xy){\rm d}x.
$$ 
For details on the Schwartz class, see \cite{StSh}. We will use the following generalized version of the Poisson summation formula.

\begin{prop}[Poisson summation] \label{Poisson} Let $\Phi:\mathbb{R}\rightarrow \mathbb{C}$ be a Schwartz class function, $L>0$ and $\alpha\in \mathbb{R}$. Then
$$
\sum\limits_{l\in \mathbb{Z}} \Phi\left(\frac{l}{L}\right) e\left(l\alpha\right)=L\sum\limits_{n\in \mathbb{Z}} \hat{\Phi}\left(L(\alpha-n)\right).
$$
\end{prop}

\begin{proof} This arises by a linear change of variables from the well-known basic version of the Poisson summation formula which asserts that
$$
\sum\limits_{n\in \mathbb{Z}} F(n)=\sum\limits_{n\in \mathbb{Z}} \hat{F}(n)
$$
for any Schwartz class function $F:\mathbb{R}\rightarrow \mathbb{C}$ (see \cite{StSh}).
\end{proof} 

Moreover, we note that by the rapid decay of $\hat\Phi$ in Proposition \ref{Poisson}, we have 
\begin{equation} \label{nsumRHS}
\sum\limits_{n\in \mathbb{Z}} \hat{\Phi}\left(L(\alpha-n)\right)\ll \chi_{[0,r^{\varepsilon}L^{-1}]}\left(||\alpha||\right)+r^{-2026}
\end{equation}
for any $r\in \mathbb{N}$, where 
$$
\chi_J(x)=\begin{cases} 1 & \mbox{ if } x\in J\\ 0, & \mbox{ otherwise} \end{cases}
$$
is the indicator function of an interval $J$. 

\subsection{Proof of Theorem \ref{bilinearbound}}
Write 
$$
\Sigma:=\Sigma(r,j,L,M,\boldsymbol{\alpha},\boldsymbol{\beta},f),
$$ 
defined as in \eqref{Sigmadeff}. 
Using the Cauchy-Schwarz inequality, we have 
\begin{equation} \label{CauSch}
| \Sigma |^2
\ll
||\boldsymbol{\alpha}||_2^2 \sum\limits_{|l|\le L} \left|\sum\limits_{1\le m\le M} \beta_m e_r\left(l\sqrt{jm}\right)e(lf(m)) \right|^2.
\end{equation}
Let $\Phi:\mathbb{R}\rightarrow \mathbb{R}_{\ge 0}$ be a Schwartz class function whose support contains $[-1,1]$. Then 
\begin{equation} \label{smoothing}
\begin{split} &
\sum\limits_{|l|\le L} \left|\sum\limits_{1\le m\le M} \beta_m e_r\left(l\sqrt{jm}\right)e(lf(m)) \right|^2 \\ \ll &
\sum\limits_{l\in \mathbb{Z}} \Phi\left(\frac{l}{L}\right)\left|\sum\limits_{1\le m\le M} \beta_m e_r\left(l\sqrt{jm}\right)e(lf(m)) \right|^2 \\
 = &
\sum\limits_{l\in \mathbb{Z}} \Phi\left(\frac{l}{L}\right)\left|\sum\limits_{1\le m\le M} \beta_m e\left(l\left(\frac{\sqrt{jm}}{r}+f(m)\right)\right) \right|^2.
\end{split}
\end{equation}
Applying Proposition \ref{Weyldif} for $H\in \mathbb{N}\cap [1,M]$, we have 
\begin{equation} \label{Weyldifapp}
\begin{split}
& \left|\sum\limits_{1\le m\le M} \beta_m e\left(l\left(\frac{\sqrt{jm}}{r}+f(m)\right)\right) \right|^2\ll  \frac{M}{H}\cdot 
\sum\limits_{1\le m\le M} |\beta_m|^2s(m)^2+\frac{M}{H}\times\\ & \sum\limits_{\substack{1\le m_1,m_2\le M\\ 1\le |m_1-m_2|\le H}} \gamma(m_1,m_2)e\left(l\left(\frac{\sqrt{jm_1}-\sqrt{jm_2}}{r}+f(m_1)-f(m_2)\right)\right),
\end{split} 
\end{equation}
where $s(m)$ is defined as in subsection 2.2 and $\gamma(m_1,m_2)$ as in \eqref{cdef}. (Note here that in the diagonal term, we need to count $m$ with multiplicity $s(m)$.) Using \eqref{sm2}, it follows that 
\begin{equation} \label{Weyldifappsimpler}
\begin{split}
& \left|\sum\limits_{1\le m\le M} \beta_m e\left(l\left(\frac{\sqrt{jm}}{r}+f(m)\right)\right) \right|^2\ll  \frac{M^2r^{\varepsilon}}{H}\cdot ||\boldsymbol{\beta}||_{\infty}^2 +\frac{M}{H}\times\\ & \sum\limits_{\substack{1\le m_1,m_2\le M\\ 1\le |m_1-m_2|\le H}} \gamma(m_1,m_2)e\left(l\left(\frac{\sqrt{jm_1}-\sqrt{jm_2}}{r}+f(m_1)-f(m_2)\right)\right).
\end{split} 
\end{equation}
Combining
\eqref{CauSch}, \eqref{smoothing} and \eqref{Weyldifappsimpler}, we deduce that
\begin{equation} \label{combisum}
\begin{split}
& | \Sigma |^2
\ll 
\frac{LM^2r^{\varepsilon}}{H}\cdot ||\boldsymbol{\alpha}||_2^2 ||\boldsymbol{\beta}||_\infty^2 + \frac{M}{H}\times\\ & \sum\limits_{\substack{1\le m_1,m_2\le M\\ 1\le |m_1-m_2|\le H}} \gamma(m_1,m_2)
\sum\limits_{l\in \mathbb{Z}} \Phi\left(\frac{l}{L}\right)
e\left(l\left(\frac{\sqrt{jm_1}-\sqrt{jm_2}}{r}+f(m_1)-f(m_2)\right)\right).
\end{split}
\end{equation}
upon redefining $\varepsilon$. 
Now applying Proposition \ref{Poisson} to the sum over $l$ above, and using \eqref{nsumRHS} and $L\le r$, we obtain
\begin{equation} \label{Poissonapp}
\begin{split}
& \sum\limits_{l\in \mathbb{Z}} \Phi\left(\frac{l}{L}\right)
e\left(l\left(\frac{\sqrt{jm_1}-\sqrt{jm_2}}{r}+f(m_1)-f(m_2)\right)\right)\\
\ll & L\chi_{[0,r^{\varepsilon}L^{-1}]}\left(\left|\left|\frac{\sqrt{jm_1}-\sqrt{jm_2}}{r}+f(m_1)-f(m_2)\right|\right|\right)+r^{-2025}.
\end{split}
\end{equation}
If $1\le m_1,m_2\le M$ and $|m_1-m_2|\le H$, then by the mean value theorem from calculus, we have
$$
|f(m_1)-f(m_2)|\le FH
$$
under the conditions in Theorem \ref{bilinearbound}. Now we restrict $H$ to the range in \eqref{Hrange} and
note that $1/LF\ge 1$ by the condition $F\le 1/L$ in Theorem \ref{bilinearbound}. It follows that 
$$
|f(m_1)-f(m_2)|\le L^{-1}.
$$
As a consequence,
\begin{equation} \label{congrurewrite}
\begin{split}
& \chi_{[0,r^{\varepsilon}L^{-1}]}\left(\left|\left|\frac{\sqrt{jm_1}-\sqrt{jm_2}}{r}+f(m_1)-f(m_2)\right|\right|\right)\\
\le & \begin{cases}  1 & \mbox{ if } \sqrt{jm_1}-\sqrt{jm_2}\equiv d \bmod{r} \mbox{ for some integer }d \mbox{ with } |d|\le 2r^{1+\varepsilon}L^{-1},\\ 0 & \mbox{ otherwise.}\end{cases}
\end{split}
\end{equation} 
Combining \eqref{combisum}, \eqref{Poissonapp} and \eqref{congrurewrite}, and using $M\le r$ and 
$$
\gamma(m_1,m_2)\le ||\boldsymbol{\beta}||_{\infty}^2,
$$ 
we get 
\begin{equation} \label{beforeenergy}
| \Sigma |^2\ll
\frac{LM^2r^{\varepsilon}}{H}\cdot ||\boldsymbol{\alpha}||_2^2 ||\boldsymbol{\beta}||_{\infty}^2 + \frac{LM}{H}\cdot ||\boldsymbol{\alpha}||_2^2 ||\boldsymbol{\beta}||_{\infty}^2 \sum\limits_{|d|\le 2r^{1+\varepsilon}L^{-1}} \sum\limits_{\substack{1\le m_1,m_2\le M\\ 1\le |m_1-m_2|\le H\\
\sqrt{jm_1}-\sqrt{jm_2}\equiv d\bmod{r}}} 1.
\end{equation}
The Cauchy-Schwarz inequality implies 
\begin{equation} \label{CSapp2}
\begin{split}
\Bigg|\sum\limits_{|d|\ll r^{1+\varepsilon}L^{-1}} \sum\limits_{\substack{1\le m_1,m_2\le M\\ 1\le |m_1-m_2|\le H\\
\sqrt{jm_1}-\sqrt{jm_2}\equiv d\bmod{r}}} 1 \Bigg|^2\ll & \frac{r^{1+\varepsilon}}{L}\cdot \sum\limits_{d=1}^r \Bigg| \sum\limits_{\substack{1\le m_1,m_2\le M\\ 1\le |m_1-m_2|\le H\\
\sqrt{jm_1}-\sqrt{jm_2}\equiv d\bmod{r}}} 1 \Bigg|^2\\
= & \frac{r^{1+\varepsilon}}{L} \cdot E_2(r,j,M,H),
\end{split}
\end{equation}
where $E_2(r,j,M,H)$ is defined as in \eqref{E2def}. Squaring \eqref{beforeenergy} and using \eqref{CSapp2}, we deduce that
$$
|\Sigma|^4 
\ll
\frac{L^2M^4r^{2\varepsilon}}{H^2}\cdot ||\boldsymbol{\alpha}||_2^4||\boldsymbol{\beta}||_{\infty}^4 + \frac{r^{1+\varepsilon}LM^2}{H^2}\cdot ||\boldsymbol{\alpha}||_2^4||\boldsymbol{\beta}||_{\infty}^4 E_2(r,j,M,H).
$$
Plugging in \eqref{E2bound} gives
$$
|\Sigma|^4 
\ll
\left(\frac{L^2M^4}{H^2}+ HLM^4+\frac{rLM^3}{H}\right)\cdot ||\boldsymbol{\alpha}||_2^4||\boldsymbol{\beta}||_{\infty}^4r^{2\varepsilon}.
$$
Taking the fourth root,
we obtain \eqref{generalbilinearbound}, which completes the proof of Theorem \ref{bilinearbound}. 

\subsection{Proof of Theorem \ref{bilinearbound2}} Our proof is similar, but here we start with an application of H\"older's inequality instead of the Cauchy-Schwarz inequality, obtaining
\begin{equation*}
| \Sigma |^4
\ll
||\boldsymbol{\alpha}||_{4/3}^{4} \sum\limits_{|l|\le L} \left|\sum\limits_{1\le m\le M} \beta_m e_r\left(l\sqrt{jm}\right)e(lf(m)) \right|^4,
\end{equation*}
where 
\begin{equation} \label{normcompare2}
||\boldsymbol{\alpha}||_{4/3}^{4}:=\left(\sum\limits_{|l|\le L} |\alpha_l|^{4/3}\right)^3\ll \left(\sum\limits_{|l|\le L} |\alpha_l|^{2}\right)^{2}L=||\boldsymbol{\alpha}||_2^{4}L,
\end{equation}
using H\"older's inequality again. Similarly as before, we have 
\begin{equation*} \label{smoothing2}
\begin{split} &
\sum\limits_{|l|\le L} \left|\sum\limits_{1\le m\le M} \beta_m e_r\left(l\sqrt{jm}\right)e(lf(m)) \right|^4 \\ \ll &
\sum\limits_{l\in \mathbb{Z}} \Phi\left(\frac{l}{L}\right)\left|\sum\limits_{1\le m\le M} \beta_m e\left(l\left(\frac{\sqrt{jm}}{r}+f(m)\right)\right) \right|^4
\end{split}
\end{equation*}
for any Schwartz class function $\Phi:\mathbb{R}\rightarrow \mathbb{R}_{\ge 0}$ with support containing $[-1,1]$. 
Squaring \eqref{Weyldifappsimpler}, we deduce that
\begin{equation*} \label{Weyldifapp2}
\begin{split}
& \left|\sum\limits_{1\le m\le M} \beta_m e\left(l\left(\frac{\sqrt{jm}}{r}+f(m)\right)\right) \right|^4\\ \ll & \frac{M^4r^{2\varepsilon}}{H^2}\cdot 
||\boldsymbol{\beta}||_{\infty}^4+\frac{M^2}{H^2}\cdot \sum\limits_{\substack{1\le m_1,m_2,m_3,m_4\le M\\ 1\le |m_1-m_2|\le H\\ 1\le |m_3-m_4|\le H}} \gamma(m_1,m_2)\overline{\gamma(m_3,m_4)}\times\\ & e\Bigg(l\Bigg(\frac{(\sqrt{jm_1}-\sqrt{jm_2})-(\sqrt{jm_3}-\sqrt{jm_4})}{r}+\\ & (f(m_1)-f(m_2))-(f(m_3)-f(m_4))\Bigg)\Bigg).
\end{split}
\end{equation*}
Applying Proposition \ref{Poisson} and \eqref{nsumRHS} in a similar way as before, we deduce that 
\begin{equation} \label{fourth}
\begin{split}
| \Sigma |^4
\ll &
\frac{LM^4r^{2\varepsilon}}{H^2}\cdot ||\boldsymbol{\alpha}||_{4/3}^{4} ||\boldsymbol{\beta}||_{\infty}^4+\frac{LM^2}{H^2}\cdot 
 ||\boldsymbol{\alpha}||_{4/3}^4 ||\boldsymbol{\beta}||_{\infty}^4 \times\\ & \sum\limits_{|d|\le 2r^{1+\varepsilon}L^{-1}} \sum\limits_{\substack{1\le m_1,m_2,m_3,m_4\le M\\ 1\le |m_1-m_2|\le H\\1\le |m_3-m_4|\le H\\
(\sqrt{jm_1}-\sqrt{jm_2})-(\sqrt{jm_3}-\sqrt{jm_4})\equiv d\bmod{r}}} 1.
\end{split}
\end{equation}
Applying the Cauchy-Schwarz inequality to the inner-most double sum on the right-hand side yields
\begin{equation} \label{CSagain}
\begin{split}
 & \Bigg| \sum\limits_{|d|\le 2r^{1+\varepsilon}L^{-1}} \sum\limits_{\substack{1\le m_1,m_2,m_3,m_4\le M\\ 1\le |m_1-m_2|\le H\\1\le |m_3-m_4|\le H\\
(\sqrt{jm_1}-\sqrt{jm_2})-(\sqrt{jm_3}-\sqrt{jm_4})\equiv d\bmod{r}}} 1 \Bigg|^2\\
\ll & \frac{r^{1+\varepsilon}}{L}\cdot E_4(r,j,M,H).
\end{split}
\end{equation}
Squaring \eqref{fourth} and using \eqref{CSagain}, we deduce that
\begin{equation*} \label{eighth}
| \Sigma |^8
\ll 
\frac{L^2M^8r^{4\varepsilon}}{H^4}\cdot ||\boldsymbol{\alpha}||_{4/3}^{8} ||\boldsymbol{\beta}||_{\infty}^8+\frac{r^{1+\varepsilon}LM^4}{H^4}\cdot 
 ||\boldsymbol{\alpha}||_{4/3}^8 ||\boldsymbol{\beta}||_{\infty}^8 E_4(r,j,M,H).
\end{equation*}
Plugging in \eqref{E4conj} gives
\begin{equation*} 
| \Sigma |^8
\ll 
\frac{L^2M^8r^{4\varepsilon}}{H^4}\cdot ||\boldsymbol{\alpha}||_{4/3}^{8} ||\boldsymbol{\beta}||_{\infty}^8+\left(H^{1-\nu}LM^8+\frac{rLM^{7}}{H^{1+\nu}}\right)r^{4\varepsilon}\cdot 
 ||\boldsymbol{\alpha}||_{4/3}^8 ||\boldsymbol{\beta}||_{\infty}^8 .
\end{equation*}
Taking \eqref{normcompare2} into account, this implies
\begin{equation*} 
| \Sigma |^8
\ll \left(
\frac{L^4M^8}{H^4}+H^{1-\nu}L^3M^{8}+\frac{rL^3M^{7}}{H^{1+\nu}}\right)\cdot  ||\boldsymbol{\alpha}||_2^8 ||\boldsymbol{\beta}||_{\infty}^8 r^{4\varepsilon}.
\end{equation*}
Now \eqref{Hypobound} follows upon taking the eighth root, completing the proof of Theorem \ref{bilinearbound2}. 

\section{Partial progress on the large sieve for square moduli}
\subsection{Review of previous work}
Extensive work on the large sieve for square moduli was carried out in \cite{BaiLS}, so we will be brief here und just indicate the arguments required to establish Theorem \ref{lsimpro}. Throughout the following, we assume that $Q^3=N\ge 1$. Also, we set
$$
\tau:=\lfloor \sqrt{N} \rfloor, \quad \Delta:=\frac{1}{N}
$$
and 
$$
P(\alpha):=\sharp\left\{(q,a)\in \mathbb{Z}^2: 1\le q\le Q, \ (q,a)=1,\ \left|\frac{a}{q^2}-\alpha\right|\le \Delta\right\}\quad \mbox{for } \alpha\in \mathbb{R}.
$$
In \cite{BaiLS}, we recalled from earlier work that for a proof of \eqref{lssq}, it suffices to establish that 
\begin{equation} \label{des}
P(x)\le Q^{1/2-\eta} 
\end{equation}
for some $\eta>0$ and all $x$ of the form
$$
x=\frac{b}{r}+z,
$$ 
where 
$$
1\le r\le \tau,\quad (r,b)=1 \quad \mbox{and} \quad \Delta\le z\le \frac{\sqrt{\Delta}}{r}.
$$
So if $Q^3=N$, then the relevant range for $r$ is 
\begin{equation*} \label{rrange}
1\le r\le Q^{3/2}.
\end{equation*}
Also, in \cite[Lemma 18]{BaiLS}, it was quoted from earlier work by the author that the estimate 
$$
P(x)\ll (1+Q^2rz+Q^3\Delta)N^{\varepsilon}
$$
holds, which turns into 
$$
P(x)\ll (1+Q^2rz)N^{\varepsilon}
$$
in our case of $Q^3=N$. This implies the desired bound \eqref{des} if 
$$
z\le \frac{\sqrt{\Delta}}{r}\cdot Q^{-\eta}.
$$
It remains to establish \eqref{des} in the $z$-range 
\begin{equation} \label{zrange}
\frac{\sqrt{\Delta}}{r}\cdot Q^{-\eta}\le z\le \frac{\sqrt{\Delta}}{r}.
\end{equation}
Here, we want to focus only on the extreme case when
\begin{equation*} \label{extreme}
z=\frac{\sqrt{\Delta}}{r}=\frac{1}{Q^{3/2}r},
\end{equation*}
which we will assume throughout the following.
It will be clear that if we obtain a power saving over $P(x)=P(b/r+z)\ll Q^{1/2}$ in this situation, then this will also work for $z$ in the range \eqref{zrange}, provided that $\eta>0$ is small enough.

 In \cite[sections 9 and 10]{BaiLS}, using bounds for exponential sums, we established such a power saving {\it unconditionally} if  
$$
1\le r\le Q^{15/26-\varepsilon}
$$
with $\varepsilon$ small enough. 
Our starting point was \cite[inequality (40)]{BaiLS}, established via Fourier analysis, which is the estimate 
\begin{equation} \label{specialPx}
P(x)\ll 1+\frac{\delta}{Qr} \left|\sum\limits_{l\in \mathbb{Z}} W\left(\frac{l\delta}{Qr}\right)\sum\limits_{m\in \mathbb{Z}}V\left(\frac{m}{Q^{1/2}}\right) e\left(-\frac{l\sqrt{m}}{r}\cdot Q^{3/4}\right) e_r\left(l\sqrt{jm}\right)\right|, 
\end{equation}
where $W:\mathbb{R}\rightarrow \mathbb{C}$ and $V:\mathbb{R}\rightarrow \mathbb{R}_{>0}$ are fixed Schwartz class functions, $V$ having compact support in $\mathbb{R}_{>0}$, and $\delta$ is a free parameter satisfying
\begin{equation} \label{deltarange}
\frac{Q^2\Delta}{z}=rQ^{1/2}\le \delta\le Q^2
\end{equation}
(see \cite[inequality (24)]{BaiLS}). Here, $\sqrt{m}$ is the ordinary square root of the positive real number $m$ in the analytic term 
\begin{equation} \label{analytic}
e\left(-\frac{l\sqrt{m}}{r}\cdot Q^{3/4}\right),
\end{equation}
and $\sqrt{jm}$ is a modular square root of $jm$ modulo $r$ in the arithmetic term
$$
 e_r\left(l\sqrt{jm}\right).
$$
To handle the remaining range 
\begin{equation} \label{remainingr}
Q^{15/26-\varepsilon}<r\le Q^{3/2}
\end{equation}
under Hypothesis \ref{hyp}, we start from \eqref{specialPx} as well. Obviously, this connects to the bilinear sums considered previously in this article. We note that bounding the right-hand side of \eqref{specialPx} trivially exactly yields the bound $P(x)\ll Q^{1/2}$ which we aim to improve.

\subsection{Application of Theorem \ref{bilinearbound2}}
As a rule of thumb, the larger $r$ is, the more difficult the problem of estimating $P(x)$ becomes. If $r$ is as large as possible, i.e.,  $r=Q^{3/2}$,  then inequality \eqref{deltarange} implies $\delta=Q^2$. In this case, the lengths of the $l-$ and $m$-summations in \eqref{specialPx} are both of size about $Q^{1/2}=r^{1/3}$, and the order of magnitude of the amplitude function in the analytic term given in \eqref{analytic} is
$$
\frac{l\sqrt{m}}{r}\cdot Q^{3/4}\ll 1.
$$
So this analytic term essentially doesn't oscillate, and we are therefore faced with a bilinear sum as in Corollaries \ref{bilinearcor} and \ref{bilinearcor2}. We argued in section 1 that if $L$ and $M$ are of size about $r^{1/3}$, then even conditional estimates for the additive energy $E_2$ do not suffice to estimate these bilinear sums non-trivially (at least, not along the lines of our method in this article). However, we were able to break this barrier under Hypothesis \ref{hyp} on the additive energy $E_4$, as demonstrated in Corollary \ref{bilinearcor2} and the following remark. Therefore, we will assume the truth of this hypothesis and directly apply Theorem \ref{bilinearbound2}.   

Using the rapid decay of $W$, we may cut off the summation over $l$ in \eqref{specialPx} at $|l|\le Q^{1+\varepsilon}r/\delta$ at the cost of a negligible error of size $O(Q^{-2026})$. Since $V$ is supposed to have compact support in $\mathbb{R}_{>0}$, the summation over $m$ can be restricted to $C_0Q^{1/2}\le m\le C_1Q^{1/2}$ for suitable constants $C_1>C_0>0$. So taking 
$$
L:=\frac{Q^{1+\varepsilon}r}{\delta}, \quad M_0:=C_0Q^{1/2} \quad \mbox{and} \quad M:=C_1Q^{1/2},
$$
it follows that 
\begin{equation*}
P(x)\ll 1+\Bigg| \frac{Q^{\varepsilon}}{L} \sum\limits_{|l|\le L} \sum\limits_{M_0\le m\le M}\alpha_l\beta_m e_r\left(l\sqrt{jm}\right) e\left(lf(m)\right)\Bigg|
\end{equation*}
for suitable $\alpha_l,\beta_m\ll 1$ and 
$$
f(x):=-\frac{Q^{3/4}\sqrt{x}}{r}. 
$$
For $M_0\le x\le M$, we have 
$$
|f'(x)|\le F:=C_2\cdot \frac{Q^{1/2}}{r} 
$$
for a suitable constant $C_2>0$. Applying Theorem \ref{bilinearbound2}, and taking $||\boldsymbol{\alpha}||_2\ll L^{1/2}$ and $\boldsymbol{\beta}||_{\infty}\ll 1$ into account, we deduce that 
\begin{equation} \label{Pxnew}
P(x)\ll 1+\left(
H^{-1/2}M+H^{1/8-\nu/8}L^{-1/8}M+H^{-1/8-\nu/8}L^{-1/8}M^{7/8}r^{1/8}\right)(Qr)^{\varepsilon},
\end{equation}
provided that 
\begin{equation} \label{MHcondis}
M\le r\quad \mbox{and} \quad 1\le H\le \min\left\{\frac{1}{LF},M\right\}.
\end{equation}

Now we choose 
\begin{equation} \label{Hfix}
H:=\frac{r^{1/2}}{Q^{1/4+2\varepsilon}}.
\end{equation}
Under this choice, the terms $H^{1/8-\nu/8}L^{-1/8}M$ and $H^{-1/8-\nu/8}L^{-1/8}M^{7/8}r^{1/8}$ on the right-hand side of \eqref{Pxnew} are nearly balanced, and the conditions in \eqref{MHcondis} are satisfied provided that
\begin{equation} \label{newrrange}
C_1Q^{1/2}\le r\le Q^{2\varepsilon}\min\left\{\frac{\delta^2}{C_2^2Q^{5/2}},C_1^2Q^{3/2}\right\}. 
\end{equation}
Using \eqref{Hfix} and $M=C_1Q^{1/2}$, \eqref{Pxnew} turns into 
\begin{equation*} 
P(x)\ll 1+\left(
Q^{5/8}r^{-1/4}+L^{-1/8}Q^{15/32+\nu/32}r^{1/16-\nu/16}\right)(Qr)^{2\varepsilon},
\end{equation*}
and using $L=Q^{1+\varepsilon}r/\delta$, it follows that
\begin{equation} \label{Pxnew2} 
P(x)\ll 1+\left(
Q^{5/8}r^{-1/4}+Q^{11/32+\nu/32}\delta^{1/8}r^{-1/16-\nu/16}\right)(Qr)^{3\varepsilon}.
\end{equation}
Now we choose 
\begin{equation*} \label{deltachois}
\delta:= r^{1/2}Q^{5/4}, 
\end{equation*}
which is consistent with \eqref{newrrange} if $Q$ is large enough. The condition \eqref{deltarange} on $\delta$ is then satisfied since $r\le Q^{3/2}$, and \eqref{Pxnew2} turns into
\begin{equation*} 
P(x)\ll \left(
Q^{5/8}r^{-1/4}+Q^{1/2+\nu/32}r^{-\nu/16}\right)(Qr)^{3\varepsilon}.
\end{equation*}
It follows that 
$$
P(x)\ll Q^{1/2+8\varepsilon-\nu^2/16},
$$
provided that $Q$ is large enough, $\nu<1$ and  
$$
Q^{1/2+\nu}\le r\le Q^{3/2}.
$$
Assuming $\nu<1/13$ without loss of generality, and taking 
$$
\varepsilon<\min\left\{\frac{\nu^2}{128},\frac{1}{13}-\nu\right\},
$$ 
this establishes a power saving of the form
$$
P(x)\ll Q^{1/2-\eta}
$$
for some $\eta>0$ if $r$ is in the range \eqref{remainingr}. Consequently, Theorem \ref{lsimpro} holds.

\section{Possible ways forward}
Below we indicate in which ways progress towards an unconditional proof of \eqref{lssq} could be made. 
\begin{itemize}
\item Firstly, it would be highly desirable to prove Hypothesis \ref{hyp}. Taking our target range \eqref{remainingr} for $r$ and our choice of $H$ in \eqref{Hfix} into account,  to establish \eqref{lssq}, we only need to cover the ranges 
$$
r^{1/3}\le M\le r^{13/15} \quad \mbox{and} \quad H\asymp r^{1/2}M^{-1/2}.
$$ 
\item If $r$ is a prime and $M\ge r^{4/7}$, we improved \eqref{T2bound} in \cite[Theorem 22]{BaiLS} using completion of exponential sums, establishing the bound
$$
T_2(r,j,M)\ll \frac{M^4}{r}+M^2+r^{3/2}. 
$$   
It should be possible to extend these techniques to cover general moduli $r$ and our energies $E_2(r,j,M,H)$ and $E_4(r,j,M,H)$. Potentially, this could shorten the problematic $r$-range.
\item A place where a serious loss occurs is \eqref{CSapp2}, where we extend the $d$-range from $|d|\le D:=r^{1+\varepsilon}/L$ to the full set of residue classes $d$ modulo $r$. Progress could also be made if we would manage to utilize the restriction $|d|\le D$ at this place. Following the proof of \eqref{E2bound} in subsection 2.2, this boils down to efficiently counting the number of solutions $(h,m,d)$ to the congruence
$$
 h^2-kd^2m+k^2d^4\equiv 0 \bmod{r}
$$
in a box
\begin{equation}
\{(h,m,d)\in \mathbb{Z}^3:  1\le h\le H, \ 1\le m\le 4M,\ |d|\le D \}.
\end{equation}
(Recall here that $k$ is a multiplicative inverse of $j$ modulo $r$ and hence fixed.) 
The critical case in our application is when $H,M\asymp r^{1/3}$ and $D\asymp r^{2/3}$. A direct treatment of the additive energy $E_4$ along the same lines as that of $E_2$ in subsection 2.2 would lead us to a congruence with a polynomial in five variables of degree 16.  
Nontrivial bounds for the number of solutions $(x_1,...,x_s,y)$ in a box 
$$
[U_1,U_1+H]\times\cdots [U_s,U_s+H]\times [V,V+D]
$$
to congruences of the form 
$$
F(x_1,...,x_s)-y\equiv 0\bmod{r}
$$
with $F$ being a polynomial in $\mathbb{Z}[X_1,...,X_s]$ were derived in \cite{Kerr}. To make progress on our problem, we would need an extension of this result to more general polynomial congruences of the form
$$
G(x_1,...,x_s,y)\equiv 0\bmod{r}
$$
with solutions in the same box.
\item Non-trivial bounds for our bilinear sums 
$$
\Sigma(r,j,L,M,\boldsymbol{\alpha}, \boldsymbol{\beta})\quad  \mbox{and} \quad \Sigma(r,j,L,M,\boldsymbol{\alpha}, \boldsymbol{\beta},f)
$$ 
for much shorter summation ranges $|l|\le L$ and $1\le m\le M$ should be achievable unconditionally for smooth and prime power moduli using analogues of the van der Corput method from the theory of exponential sums. However, to tackle the large sieve with square moduli, we need to make unconditional progress for the {\it full} set of moduli $1\le r\le N^{1/2}$. If we restrict this set of moduli, we run into trouble when approximating real numbers $x$ by rational numbers $b/r$ in the underlying method of Wolke, outlined in \cite[section 4]{BaiLS}. This is because the quality of Diophantine approximation deteriorates.    
\item Our method might be adaptable to the situation when both variables come from modular roots. This would have potential applications to generalising Dunn's and Zaharescu's recent work \cite{DuZa} on the twisted second moment of modular half integral weight $L$-functions from prime to arbitrary moduli. 
\end{itemize}


\begin{thebibliography}{20}
\bibitem{BagShp} N. Bag; I.E. Shparlinski, {\it Bounds of some double exponential sums},
J. Number Theory 219, 228--236 (2021).

\bibitem{BaiLS} S. Baier, {\it The large sieve for square moduli, revisited}, Preprint, arXiv:2503.18009 (2026), to appear in Hardy-Ramanujan J., Vol. 48.

\bibitem{BaiZhao} S. Baier; L. Zhao, {\it 
Bombieri-Vinogradov type theorems for sparse sets of moduli},
Acta Arith. 125, No. 2, 187--201 (2006).

\bibitem{BaiZhao1} S. Baier; L. Zhao, {\it 
An improvement for the large sieve for square moduli},
J. Number Theory 128, No. 1, 154--174 (2008).

\bibitem{BPS} W.D. Banks; F. Pappalardi; I.E. Shparlinski, {\it 
On group structures realized by elliptic curves over arbitrary finite fields},
Exp. Math. 21, No. 1, 11--25 (2012).

\bibitem{BFKS} J. Bourgain; K. Ford; S.V. Konyagin; I.E. Shparlinski, {\it 
On the divisibility of Fermat quotients},
Mich. Math. J. 59, No. 2, 313--328 (2010).

\bibitem{BHW} U. Betke; M. Henk; J. M. Wills, {\it Successive-minima-type inequalities}, Discr. Comput.
Geom. 9, 165–175 (1993).

\bibitem{CiShZ} J. Cilleruelo; I.E. Shparlinski; A. Zumalac\'arregui, {\it Isomorphism classes of elliptic curves over a finite field in some thin families}, Math. Res. Lett. 19, 335--343 (2012).

\bibitem{DKSZ} A. Dunn; B. Kerr; I.E. Shparlinski; A. Zaharescu, {\it
Bilinear forms in Weyl sums for modular square roots and applications},
Adv. Math. 375, Article ID 107369, 58 p. (2020).

\bibitem{DuZa} A. Dunn; A. Zaharescu, {\it
The twisted second moment of modular half-integral weight L-functions},
J. Eur. Math. Soc. (JEMS) 27, No. 1, 1--69 (2025).

\bibitem{KSXW} B. Kerr; I.E. Shparlinski; X. Wu; P. Xi, {\it
Bounds on bilinear forms with Kloosterman sums},
J. Lond. Math. Soc., II. Ser. 108, No. 2, 578--621 (2023).

\bibitem{GrKo} S.W. Graham; G. Kolesnik, {\it 
Van der Corput’s method for exponential sums},
London Mathematical Society Lecture Note Series, 126. Cambridge etc.: Cambridge University Press. 120 p.

\bibitem{Kerr} B. Kerr, {\it 
Solutions to polynomial congruences in well-shaped sets},
Bull. Aust. Math. Soc. 88, No. 3, 435--447 (2013).

\bibitem{KSSZ} B. Kerr; I.D. Shkredov; I.E. Shparlinski; A. Zaharescu, {\it 
Energy bounds for modular roots and their applications}, Preprint, arXiv:2103.09405 [math.NT] (2021).

\bibitem{Mat} K. Matom\"aki, {\it
A note on primes of the form $p=aq^2+1$},
Acta Arith. 137, No. 2, 133--137 (2009).

\bibitem{Mer} J. Merikoski, {\it 
On the largest square divisor of shifted primes},
Acta Arith. 196, No. 4, 349--386 (2020).

\bibitem{SSZ22}  I.D. Shkredov; I.E. Shparlinski; A. Zaharescu, {\it Bilinear forms with modular square roots and twisted second moments of half integral weight Dirichlet series}, Int. Math. Res. Not.,  no. 22, 17431--17474 (2022).

\bibitem{SSZ} I.D. Shkredov; I.E. Shparlinski; A. Zaharescu, {\it 
On the distribution of modular square roots of primes},
Math. Z. 306, No. 3, Paper No. 43, 17 p. (2024).

\bibitem{SZ} I.E. Shparlinski; L. Zhao, {\it 
Elliptic curves in isogeny classes},
J. Number Theory 191, 194-212 (2018).

\bibitem{StSh} E.M. Stein; R. Shakarchi, {\it 
Real analysis. Measure theory, integration, and Hilbert spaces},
Princeton Lectures in Analysis 3. Princeton, NJ: Princeton University Press. xix, 402 p. (2005).

\bibitem{TaoVu} T. Tao; V. Vu, {\it Additive Combinatorics}, Cambridge, Stud. Adv. Math. 105, Cambridge
Univ. Press, Cambridge, 2006.

\bibitem{TuPa} A. Tuxanidy; D. Panario, {\it 
Infinitude of palindromic almost-prime numbers},
Int. Math. Res. Not. 2024, No. 18, 12466--12503 (2024).

\bibitem{Zhao1} L. Zhao, {\it 
Large sieve inequality with characters to square moduli}, 
Acta Arith. 112, No. 3, 297-308 (2004).
\end{thebibliography}
\end{document}